\documentclass[10pt]{amsart}
\usepackage{amssymb,amstext,amsmath,amscd,amsthm,amsfonts,enumerate,graphicx,latexsym,stmaryrd,ascmac,url}
\usepackage[usenames]{color}
\usepackage[all]{xy}

\newtheorem{thm}{Theorem}[section]
\newtheorem{lem}[thm]{Lemma}
\newtheorem{prop}[thm]{Proposition}
\newtheorem{cor}[thm]{Corollary}
\theoremstyle{definition}

\newtheorem{rem}[thm]{Remark}

\numberwithin{equation}{thm}

\def\add{\operatorname{\mathsf{add}}}
\def\C{\mathbb{C}}
\def\cm{\operatorname{\mathsf{CM}}}
\def\im{\operatorname{Im}}
\def\lcm{\operatorname{\mathsf{\underline{CM}}}}
\def\cok{\operatorname{Cok}}

\def\depth{\operatorname{depth}}
\def\ds{\operatorname{\mathsf{D_{sg}}}}
\def\inc{\mathsf{inc}}
\def\ker{\operatorname{Ker}}
\def\mf{\operatorname{\mathsf{MF}}}
\def\mod{\operatorname{\mathsf{mod}}}

\def\P{\mathbf{P}}
\def\pd{\operatorname{pd}}
\def\radius{\operatorname{radius}}

\def\size{\operatorname{size}}
\def\syz{\Omega}
\def\t{{}^\mathrm{t}\!}
\def\X{\mathcal{X}}
\def\Y{\mathcal{Y}}
\def\Z{\mathcal{Z}}

\begin{document}
\allowdisplaybreaks
\title[On the radius of the category of extensions of matrix factorizations]{On the radius of the category of extensions\\of matrix factorizations}
\author{Kaori Shimada}
\address[K. Shimada]{Department of Mathematics, School of Science and Technology, Meiji University, 1-1-1 Higashi-mita, Tama-ku, Kawasaki 214-8571, Japan}
\email{k\_shimada@meiji.ac.jp}
\author{Ryo Takahashi}
\address[R. Takahashi]{Graduate School of Mathematics, Nagoya University, Furocho, Chikusaku, Nagoya 464-8602, Japan\,/\,Department of Mathematics, University of Kansas, Lawrence, KS 66045-7523, USA}
\email{takahashi@math.nagoya-u.ac.jp}
\urladdr{https://www.math.nagoya-u.ac.jp/~takahashi/}
\subjclass[2010]{13C14, 13C60, 13D09}
\keywords{dimension, hypersurface, matrix factorization, maximal Cohen--Macaulay module, radius, singularity category}
\thanks{Ryo Takahashi was partly supported by JSPS Grant-in-Aid for Scientific Research 16K05098, 19K03443 and JSPS Fund for the Promotion of Joint International Research 16KK0099}
\begin{abstract}
Let $S$ be a commutative noetherian ring.
The extensions of matrix factorizations of non-zerodivisors $x_1,\dots,x_n$ of $S$ form a full subcategory of finitely generated modules over the quotient ring $S/(x_1\cdots x_n)$.
In this paper, we investigate the radius (in the sense of Dao and Takahashi) of this full subcategory.
As an application, we obtain an upper bound of the dimension (in the sense of Rouquier) of the singularity category of a local hypersurface of dimension one, which refines a recent result of Kawasaki, Nakamura and Shimada.
\end{abstract}
\maketitle
\section{Introduction}

Rouquier \cite{R} has introduced the notion of the dimension of a triangulated category.
As an analogue for abelian categories, Dao and Takahashi \cite{radius,dim} have introduced the notions of the dimension and radius of a full subcategory of an abelian category with enough projective objects.
This paper studies the dimension and radius of a full subcategory of the category of finitely generated modules over a commutative noetherian ring, and the dimension of the singularity category of a commutative noetherian ring.

To explain our results more precisely, let $R$ be a commutative noetherian ring.
Denote by $\mod R$ the category of finitely generated $R$-modules, and by $\cm(R)$ the full subcategory of $\mod R$ consisting of maximal Cohen--Macaulay modules.
Kawasaki, Nakamura and Shimada \cite{KNS} have recently investigated the dimension of $\cm(R)$ in the case where $R$ is a certain hypersurface of dimension one.
The main purpose of this paper is to develop a further studies of this theorem.

Let $S$ be a commutative noetherian ring and $x\in S$.
Denote by $\mf(x)$ the full subcategory of $\mod S$ consisting of modules $M$ with $xM=0$ admitting an exact sequence of the form $0\to S^n\to S^m\to M\to0$.
Note that $\mf(x)$ is regarded as a full subcategory of $\mod S/(x)$.
In the case where $x$ is a non-zerodivisor, $\mf(x)$ coincides with the category of matrix factorizations of $x$ over $S$; see Proposition \ref{7}.

For ideals $I,J$ of $S$ and full subcategories $\X,\Y$ of $\mod S/I,\mod S/J$ respectively, we denote by $\X\ast\Y$ the full subcategory of $\mod S/IJ$ consisting of modules $M$ admitting an exact sequence $0\to X\to M\to Y\to0$ with $X\in\X$ and $Y\in\Y$.
The operation $-\ast-$ satisfies the associativity; see Proposition \ref{8}.

The main result of this paper is the following theorem.

\begin{thm}\label{0}
Let $S$ be a commutative noetherian ring and $x_1,\dots,x_n\in S$ non-zerodivisors.
Then
$$
\radius(\mf(x_1)\ast\cdots\ast\mf(x_n))\le\sup\{\dim\mf(x_1),\dots,\dim\mf(x_n)\}+1.
$$
\end{thm}

For a noetherian ring $R$ we denote by $\ds(R)$ the singularity category of $R$, i.e., the Verdier quotient of the bounded derived category of $\mod R$ by perfect complexes.
The above theorem yields the following corollary, which gives rise to an inequality of the dimensions of the singularity categories of $1$-dimensional hypersurfaces.
This corollary refines a recent result of Kawasaki, Nakamura and Shimada \cite[Theorem 4.5]{KNS}, which assumes that the elements $x_1,\dots,x_n$ are powers of distinct prime elements and that the local ring $S$ is complete.

\begin{cor}\label{1}
Let $S$ be a regular local ring of dimension two and $x_1,\dots,x_n\in S$.
Then one has
$$
\dim\ds(S/(x_1\cdots x_n))\le\sup_{1\le i\le n}\{\dim\ds(S/(x_i))\}+1.
$$
In particular, if $S/(x_i)$ has finite CM-representation type for $1\le i\le n$, then $\dim\ds(S/(x_1\cdots x_n))\le1$.
\end{cor}

Here we introduce a set of polynomials over $\C$:
$$
\P=\{x,y,x^2+y^{m+1},x^2y+y^{n-1},x^3+y^4,x^3+xy^3,x^3+y^5\mid m\ge1,n\ge4\}.
$$
The inequality of dimensions of singularity categories given in the above result implies the following.

\begin{cor}\label{2}
Let $f_1,\dots,f_r\in\P$ and $R=\C[\![x,y]\!]/(f_1\cdots f_r)$.
Then one has $\dim\ds(R)\le1$.
Moreover, $\dim\ds(R)=1$ if and only if $R$ is not isomorphic to $\C[\![x,y]\!]/(f)$ for all $f\in\P$.
\end{cor}

Proofs of the three results stated above are given in the next section.

\section{Proofs of our results}

Throughout the section, let $R$ and $S$ be commutative noetherian rings.
We assume that all modules are finitely generated, and all subcategories are full.
We denote by $E$ (resp. $E_n$) an identity matrix of some size (resp. the identity matrix of size $n$).

Let $A$ be an $m\times n$ matrix over $S$.
We define $\ker A$, $\im A$ and $\cok A$ by the kernel, image and cokernel of the linear map $A:S^n\to S^m$.
We call $A$ a {\em presentation matrix} of an $S$-module $M$ if $\cok A\cong M$.
For an $R$-module $M$ and an integer $n\ge0$ we denote by $\syz^nM$ (or $\syz_R^nM$) the {\em $n$th syzygy} of $M$, that is, the image of the $n$th differential map in a projective resolution of $M$.
This is uniquely determined up to projective summands.
We investigate the category of matrix factorizations of a non-zerodivisor.

\begin{prop}\label{7}
Let $x\in S$ be an $S$-regular element.
\begin{enumerate}[\rm(1)]
\item
Let $A,B$ be $n\times n$ matrices over $S$ such that $AB=BA=xE$.
Then $\ker A=0$ and $\cok A\in\mf(x)$.
\item
Let $M\in\mf(x)$.
Then there exist square matrices $A,B$ over $S$ with $AB=BA=xE$ and $\cok A\cong M$.
\item
Let $A,B$ be $n\times n$ matrices over $S$ with $AB=BA=xE_n$, and set $M=\cok A$.
Then the sequence
$$
\cdots\xrightarrow{A}(S/(x))^n\xrightarrow{B}(S/(x))^n\xrightarrow{A}(S/(x))^n\xrightarrow{B}\cdots
$$
and its $S/(x)$-dual are both exact sequences.
In particular, $\syz_{S/(x)}^2M\cong M$.
\item
If $S$ is a Cohen--Macaulay local ring, then $\mf(x)\subseteq\cm(S/(x))$.
\item
If $S$ is a regular local ring, then $\mf(x)=\cm(S/(x))$.
\end{enumerate}
\end{prop}

\begin{proof}
(1) There is an exact sequence $S^n\xrightarrow{A}S^n\to\cok A\to0$ of $S$-modules.
As $BA=xE_n$ and $x$ is $S$-regular, it is seen that the map $A:S^n\to S^n$ is injective, or in other words, $\ker A=0$.
Since $AB=xE_n$, it is observed that $x$ annihilates $\cok A$.
Hence $\cok A$ belongs to $\mf(x)$.

(2) By definition, $x$ kills $M$ and there is an exact sequence $0\to S^n\xrightarrow{A}S^m\to M\to0$ of $S$-modules.
As $x$ is $S$-regular, we see that $M$ has rank $0$ as an $S$-module, which implies $m=n$.
Since $xM=0$, as we see in the commutative diagram below with exact rows, there is an $n\times n$ matrix $B$ such that $AB=xE_n$.
$$
\xymatrix@R-.5pc@C+2pc{
0\ar[r] & S^n\ar[r]^A\ar[d]_x & S^n\ar[r]\ar[d]_x\ar@{.>}[ld]_B\ar[rd]_0 & M\ar[r]\ar[d]_x^0 & 0\\
0\ar[r] & S^n\ar[r]^A & S^n\ar[r] & M\ar[r] & 0
}
$$
The above diagram also says that $BA=xE_n$, and the assertion follows.

(3) The equality $AB=xE_n$ implies that the sequence $(S/(x))^n\xrightarrow{B}(S/(x))^n\xrightarrow{A}(S/(x))^n$ is a complex.
Let $z\in S^n$ be an element whose residue class $\overline z\in(S/(x))^n$ satisfies $A\overline z=0$.
Then $Az\in xS^n$, and we have $xz=BAz\in BxS^n=xBS^n$.
Since $x$ is an $S$-regular element, $z$ belongs to $BS^n$.
Hence the sequence $(S/(x))^n\xrightarrow{B}(S/(x))^n\xrightarrow{A}(S/(x))^n$ is exact.
A symmetric argument shows that the sequence $(S/(x))^n\xrightarrow{A}(S/(x))^n\xrightarrow{B}(S/(x))^n$ is also exact.
Thus we obtain an exact sequence
$$
\cdots\xrightarrow{A}(S/(x))^n\xrightarrow{B}(S/(x))^n\xrightarrow{A}(S/(x))^n\xrightarrow{B}\cdots.
$$
Applying the transpose $\t\,(-)$ to the equalities $AB=BA=xE_n$ of matrices, we get the equality $\t A\,\t B=\t B\,\t A=xE_n$.
Hence the sequence
$$
\cdots\xrightarrow{\t A}(S/(x))^n\xrightarrow{\t B}(S/(x))^n\xrightarrow{\t A}(S/(x))^n\xrightarrow{\t B}\cdots
$$
is exact as well, which is nothing but the $S/(x)$-dual of the previous exact sequence.

(4) Let $M\in\mf(x)$.
Then $M$ is a module over $S/(x)$, and has projective dimension at most one as a module over $S$.
Using the Auslander--Buchsbaum formula, we get $\depth M\ge\depth S-1=\dim S-1=\dim S/(x)$.
It follows that $M$ is a maximal Cohen--Macaulay $S/(x)$-module.

(5) Let $M\in\cm(S/(x))$.
Then $\depth M\ge\dim S/(x)=\dim S-1=\depth S-1$.
Since $S$ is regular, $M$ has finite projective dimension.
Hence $\pd_SM=\depth S-\depth M\le1$, and there is an exact sequence $0\to S^n\to S^m\to M\to0$.
Thus $\cm(S/(x))\subseteq\mf(x)$.
The opposite inclusion follows from (4).
\end{proof}

In the next proposition, we verify that the operation $-\ast-$ satisfies the associativity.
Thanks to this proposition, we may use the notation $\X_1\ast\X_2\ast\cdots\ast\X_n$ without caring about any confusion, where $I_1,\dots,I_n$ are ideals of $S$ and $\X_1,\dots,\X_n$ are subcategories of $\mod S/I_1,\dots,\mod S/I_n$ respectively.

\begin{prop}\label{8}
Let $\X,\Y,\Z$ be subcategories of $\mod S/I,\mod S/J,\mod S/K$ respectively.
Then there is an equality $(\X\ast\Y)\ast\Z=\X\ast(\Y\ast\Z)$ of subcategories of $\mod S/IJK$.
\end{prop}

\begin{proof}
Let $M$ be an $S/IJK$-module.
Suppose that $M$ belongs to $(\X\ast\Y)\ast\Z$.
Then there is an exact sequence $0\to N\to M\to Z\to0$ such that $N\in\X\ast\Y$ and $Z\in\Z$.
Hence there is an exact sequence $0\to X\to N\to Y\to0$ with $X\in\X$ and $Y\in\Y$.
We make a pushout diagram:
$$
\xymatrix@R-1pc@C+2pc{
& & 0\ar[d] & 0\ar[d] \\
0\ar[r] & X\ar[r]\ar@{=}[d] & N\ar[r]\ar[d] & Y\ar[r]\ar[d] & 0 \\
0\ar[r] & X\ar[r] & M\ar[r]\ar[d] & L\ar[r]\ar[d] & 0 \\
& & Z\ar@{=}[r]\ar[d] & Z\ar[d] \\
& & 0 & 0
}
$$
The second column shows that $L$ is in $\Y\ast\Z$.
The second row implies that $M$ belongs to $\X\ast(\Y\ast\Z)$.
Thus, the inclusion $(\X\ast\Y)\ast\Z\subseteq\X\ast(\Y\ast\Z)$ follows.
The opposite inclusion is proved by a dual argument.
\end{proof}

From now on, we establish a couple of lemmas to prove our main results.

\begin{lem}\label{3}
Let $x_1,\dots,x_n\in S$ with $n\ge1$.
Let $M$ be an $S$-module, and let $0=M_0\subseteq M_1\subseteq\cdots\subseteq M_n=M$ be a filtration of $S$-submodules of $M$.
For each $1\le i\le n$, let $A_i$ be a presentation matrix of the $S$-module $M_i/M_{i-1}$, and assume $x_i(M_i/M_{i-1})=0$.
If $x_i$ is $S$-regular and $\ker A_i=0$ for all $2\le i\le n$, then there exists an exact sequence of the form
$$
0 \to \bigoplus_{i=1}^n\cok(x_1\cdots x_{i-1}A_i) \to M\oplus(S/(x_1\cdots x_n))^p \to \bigoplus_{i=2}^n(S/(x_i\cdots x_n))^{p_i} \to 0.
$$
\end{lem}

\begin{proof}
The assertion is easy to check for $n=1$.
Let $n\ge2$.
For each $2\le i\le n$, the element $x_i$ is regular and annihilates $M_i/M_{i-1}$, whence the $S$-module $M_i/M_{i-1}$ has rank $0$.
There are exact sequences
$$
S^{q_1}\xrightarrow{A_1}S^{p_1}\to M_1\to0,\qquad
0 \to S^{p_i} \xrightarrow{A_i} S^{p_i} \to M_i/M_{i-1} \to 0\quad(2\le i\le n).
$$
The multiplications by $x_1,\dots,x_n$ induce the chain maps below.
Since $x_i(M_i/M_{i-1})=0$ for all $1\le i\le n$, similarly as in the proof of Proposition \ref{7}(2) and as explained in the diagram below, there exist matrices $B_1,B_2,\dots,B_n$ such that $A_1B_1=x_1E$ and $A_iB_i=B_iA_i=x_iE$ for all $2\le i\le n$.
$$
\xymatrix@R-.5pc@C+1pc{
S^{q_1}\ar[r]^{A_1}\ar[d]_{x_1} & S^{p_1}\ar[r]\ar[d]_{x_1}\ar@{.>}[ld]_{B_1} & M_1\ar[r]\ar[d]_{x_1}^0 & 0\\
S^{q_1}\ar[r]^{A_1} & S^{p_1}\ar[r] & M_1\ar[r] & 0
}
\qquad
\xymatrix@R-.5pc@C+1pc{
0\ar[r] & S^{p_i}\ar[r]^{A_i}\ar[d]_{x_i} & S^{p_i}\ar[r]\ar[d]_{x_i}\ar@{.>}[ld]_{B_i} & M_i/M_{i-1}\ar[r]\ar[d]_{x_i}^0 & 0\\
0\ar[r] & S^{p_i}\ar[r]^{A_i} & S^{p_i}\ar[r] & M_i/M_{i-1}\ar[r] & 0
}
$$
A repeated application of the horseshoe lemma gives an exact sequence
$$
S^{q_1}\oplus S^{p_2}\oplus\cdots\oplus S^{p_n}\xrightarrow{A}S^{p_1}\oplus S^{p_2}\oplus\cdots\oplus S^{p_n}\to M\to0
$$
of $S$-module, where $A=\left(\begin{smallmatrix}
A_1&A_{12}&\cdots&A_{1n}\\
&A_2&\cdots&A_{2n}\vspace{-5pt}\\
&&{\tiny\ddots}&{\tiny\vdots}\\
&&&A_n
\end{smallmatrix}\right)$.
There are equivalences of matrices over $S/(x_1\cdots x_n)$.
\begin{align*}
C:=&\left(\begin{smallmatrix}
A_1&&&&&A_{12}&A_{13}&\cdots&A_{1,n-1}&A_{1n}\\
&x_1A_2&&&&A_2&A_{23}&\cdots&A_{2,n-1}&A_{2n}\vspace{-5pt}\\
&&{\tiny\ddots}&&&&&&{\tiny\vdots}&{\tiny\vdots}\\
&&&x_1\cdots x_{n-2}A_{n-1}&&&&&A_{n-1}&A_{n-1,n}\\
&&&&x_1\cdots x_{n-1}A_n&&&&&A_n\\
&&&&&x_2\cdots x_nE\\
&&&&&&x_3\cdots x_nE\vspace{-5pt}\\
&&&&&&&{\tiny\ddots}\\
&&&&&&&&x_{n-1}x_nE\\
&&&&&&&&&x_nE
\end{smallmatrix}\right)\\
\cong
&\left(\begin{smallmatrix}
A_1&&&&-x_1\cdots x_{n-1}A_{1n}&A_{12}&A_{13}&\cdots&A_{1,n-1}&A_{1n}\\
&x_1A_2&&&-x_1\cdots x_{n-1}A_{2n}&A_2&A_{23}&\cdots&A_{2,n-1}&A_{2n}\vspace{-5pt}\\
&&{\tiny\ddots}&&{\tiny\vdots}&&&&{\tiny\vdots}&{\tiny\vdots}\\
&&&x_1\cdots x_{n-2}A_{n-1}&-x_1\cdots x_{n-1}A_{n-1,n}&&&&A_{n-1}&A_{n-1,n}\\
&&&&0&&&&&A_n\\
&&&&&x_2\cdots x_nE\\
&&&&&&x_3\cdots x_nE\vspace{-5pt}\\
&&&&&&&{\tiny\ddots}\\
&&&&&&&&x_{n-1}x_nE\\
&&&&&&&&&x_nE
\end{smallmatrix}\right)\\
\cong
&\left(\begin{smallmatrix}
A_1&&&&0&A_{12}&A_{13}&\cdots&A_{1,n-1}&A_{1n}\\
&x_1A_2&&&0&A_2&A_{23}&\cdots&A_{2,n-1}&A_{2n}\vspace{-5pt}\\
&&{\tiny\ddots}&&{\tiny\vdots}&&&&{\tiny\vdots}&{\tiny\vdots}\\
&&&x_1\cdots x_{n-2}A_{n-1}&0&&&&A_{n-1}&A_{n-1,n}\\
&&&&0&&&&&A_n\\
&&&&&x_2\cdots x_nE\\
&&&&&&x_3\cdots x_nE\vspace{-5pt}\\
&&&&&&&{\tiny\ddots}\\
&&&&&&&&x_{n-1}x_nE\\
&&&&&&&&&x_nE
\end{smallmatrix}\right)\\
\cong
&\left(\begin{smallmatrix}
A_1&&&&&A_{12}&A_{13}&\cdots&A_{1,n-1}&A_{1n}\\
&0&&&&A_2&A_{23}&\cdots&A_{2,n-1}&A_{2n}\vspace{-5pt}\\
&&{\tiny\ddots}&&&&&&{\tiny\vdots}&{\tiny\vdots}\\
&&&0&&&&&A_{n-1}&A_{n-1,n}\\
&&&&0&&&&&A_n\\
&&&&&x_2\cdots x_nE\\
&&&&&&x_3\cdots x_nE\vspace{-5pt}\\
&&&&&&&{\tiny\ddots}\\
&&&&&&&&x_{n-1}x_nE\\
&&&&&&&&&x_nE
\end{smallmatrix}\right)\\
\cong&
\left(\begin{smallmatrix}
A_1&&&&&A_{12}&A_{13}&\cdots&A_{1,n-1}&A_{1n}\\
&0&&&&A_2&A_{23}&\cdots&A_{2,n-1}&A_{2n}\vspace{-5pt}\\
&&{\tiny\ddots}&&&&&&{\tiny\vdots}&{\tiny\vdots}\\
&&&0&&&&&A_{n-1}&A_{n-1,n}\\
&&&&0&&&&&A_n\\
&&&&&0&-x_3\cdots x_nB_2A_{23}&\cdots&-x_3\cdots x_nB_2A_{2,n-1}&-x_3\cdots x_nB_2A_{2n}\\
&&&&&&x_3\cdots x_nE\vspace{-5pt}\\
&&&&&&&{\tiny\ddots}\\
&&&&&&&&x_{n-1}x_nE\\
&&&&&&&&&x_nE
\end{smallmatrix}\right)\\
\cong&
\left(\begin{smallmatrix}
A_1&&&&&A_{12}&A_{13}&\cdots&A_{1,n-1}&A_{1n}\\
&0&&&&A_2&A_{23}&\cdots&A_{2,n-1}&A_{2n}\vspace{-5pt}\\
&&{\tiny\ddots}&&&&&&{\tiny\vdots}&{\tiny\vdots}\\
&&&0&&&&&A_{n-1}&A_{n-1,n}\\
&&&&0&&&&&A_n\\
&&&&&0&0&\cdots&0&0\\
&&&&&&x_3\cdots x_nE\vspace{-5pt}\\
&&&&&&&{\tiny\ddots}\\
&&&&&&&&x_{n-1}x_nE\\
&&&&&&&&&x_nE
\end{smallmatrix}\right)
\cong
\left(\begin{smallmatrix}
A_1&&&&&A_{12}&A_{13}&\cdots&A_{1,n-1}&A_{1n}\\
&0&&&&A_2&A_{23}&\cdots&A_{2,n-1}&A_{2n}\vspace{-5pt}\\
&&{\tiny\ddots}&&&&&&{\tiny\vdots}&{\tiny\vdots}\\
&&&0&&&&&A_{n-1}&A_{n-1,n}\\
&&&&0&&&&&A_n\\
&&&&&0\\
&&&&&&0\vspace{-5pt}\\
&&&&&&&{\tiny\ddots}\\
&&&&&&&&0\\
&&&&&&&&&0
\end{smallmatrix}\right)
\cong
\begin{pmatrix}
A&0\\
0&0
\end{pmatrix}.
\end{align*}
Here, the first equivalence follows from multiplying the last (i.e. $(2n-1)$st) block column by $-x_1\cdots x_{n-1}E$ and adding it to the $n$th block column; note that $x_1\cdots x_n=0$ in $S/(x_1\cdots x_n)$.
The second equivalence is obtained by multiplying the $i$th block column by $B_ix_{i+1}\cdots x_{n-1}A_{in}$ from the right and adding it to the $n$th block column for each $1\le i\le n-1$.
Iteraing this procedure on the $(2n-1)$st and $n$th block columns for the $(2n-1-i)$th and $(n-i)$th block columns with $1\le i\le n-2$, we get the third equivalence.
The fourth equivalence follows from multiplying the $2$nd block row by $-B_2x_3\cdots x_nE$ from the left and adding it to the $(n+1)$st block row.
The fifth equivalence is obtained by multiplying the $(n+i)$th block row by $B_2x_3\cdots x_iA_{2,i+1}$ from the left and adding it to the $(n+1)$st block row for each $2\le i\le n-1$.
Iteraing this procedure on the $2$nd and $(n+1)$st block rows for the $i$th and $(n+i-1)$st block rows with $3\le i\le n$, we get the sixth equivalence.
Replacing block columns gives the final seventh equivalence.

By assumption, $x_ix_{i+1}\cdots x_n$ is a regular element for $2\le i\le n$.
There is a commutative diagram
$$
\xymatrix{
0\ar[r] & S^{q_1+p_2+\cdots+p_n}\ar[r]\ar[d]_D & S^{q_1+p_2+\cdots+p_n}\oplus S^{p_2+\cdots+p_n} \ar[r]\ar[d]_C & S^{p_2+\cdots+p_n}\ar[r]\ar@{^(->}[d]_F & 0\\
0\ar[r] & S^{p_1+p_2+\cdots+p_n}\ar[r] & S^{p_1+p_2+\cdots+p_n}\oplus S^{p_2+\cdots+p_n}\ar[r] & S^{p_2+\cdots+p_n}\ar[r] & 0
}
$$
with exact rows, where $D=A_1\oplus x_1A_2\oplus\cdots\oplus(x_1\cdots x_{n-1})A_n$ and $F=(x_2\cdots x_n)E_{p_2}\oplus(x_3\cdots x_n)E_{p_3}\oplus\cdots\oplus x_{n-1}x_nE_{p_{n-1}}\oplus x_nE_{p_n}$; note that the map $F$ is injective.
The snake lemma yields an exact sequence
$$\textstyle
0 \to \bigoplus_{i=1}^n\cok(x_1\cdots x_{i-1}A_i) \to M\oplus(S/(x_1\cdots x_n))^p \to \bigoplus_{i=2}^n(S/(x_i\cdots x_n))^{p_i} \to 0,
$$
where we set $p=p_2+\cdots+p_n$.
Thus the proof of the lemma is completed.
\end{proof}

To state the next two lemmas, we need to recall some notation.
Let $\X,\Y$ be subcategories of $\mod R$.
Let $M$ be an $R$-module, and let $r$ be a positive integer.

\begin{enumerate}[(a)]
\item
The {\em additive closure} $\add\X$ of $\X$ is by definition the subcategory of $\mod R$ consisting of direct summands of finite direct sums of objects in $\X$.
We put $|\X|=\add\X$ and $|M|=|\{M\}|$.
\item
We denote by $[\X]$ the additive closure of the subcategory of $\mod R$ consisting of $R$ and all modules of the form $\syz^iX$, where $i\ge0$ and $X\in\X$.
We set $[M]=[\{M\}]$.
\item
We denote by $\X\circ\Y$ the subcategory of $\mod R$ consisting of the $R$-modules $E$ appearing in exact sequences of the form $0 \to X \to E \to Y \to 0$ with $X\in\X$ and $Y\in\Y$.
\item
We define
$$
{[\X]}_r=
\begin{cases}
[\X] & (r=1),\\
[{[\X]}_{r-1}\circ[\X]] & (r\ge2).
\end{cases}
\qquad
{|\X|}_r=
\begin{cases}
|\X| & (r=1),\\
|{|\X|}_{r-1}\circ|\X|| & (r\ge2).
\end{cases}
$$
\end{enumerate}
We write ${[\X]}^R,{[\X]}_r^R,{|\X|}^R,{|\X|}_r^R$ to specify the ground ring.
We set $[M]_r={[\{M\}]}_r$ and ${|M|}_r={|\{M\}|}_r$.

The following elementary remark is necessary in the proof of the first lemma.

\begin{rem}\label{9}
Let $F,G$ be $S$-modules, and let $M,N$ be submodules of $F,G$ respectively.
Let $x$ be an element of $S$.
Suppose that there is a commutative diagram of $S$-modules in the lower left whose vertical arrows are isomorphisms and horizontal arrows are inclusion maps.
Then one has a commutative diagram in the lower right, which induces an isomorphism $F/xM\cong G/xN$.
$$
\xymatrix@R-.5pc@C+2pc{
M\ar[r]^\inc\ar[d]^\cong & F\ar[d]^\cong\\
N\ar[r]^\inc & G
}
\qquad\qquad
\xymatrix@R-.5pc@C+2pc{
xM\ar[r]^\inc\ar[d]^\cong & xF\ar[d]^\cong\ar[r]^\inc & F\ar[d]^\cong\\
xN\ar[r]^\inc & xG\ar[r]^\inc & G
}
$$
\end{rem}

Now we can state those two lemmas.

\begin{lem}\label{4}
Let $A$ be a matrix over $S$.
\begin{enumerate}[\rm(1)]
\item
Let $B$ be a matrix over $S$, and let $x\in S$.
If $\cok B$ is a direct summand of $\cok A$, then $\cok(xB)$ is a direct summand of $\cok(xA)\oplus(S/(x))^n$ for some $n\ge0$.
\item
Let $x,y\in S$.
If $y\cdot\cok A=0$, then $xy\cdot\cok(xA)=0$.
In other words, if $\cok A$ is an $S/(y)$-module, then $\cok(xA)$ is an $S/(xy)$-module.
\item
Let $y\in S$.
Assume that $\ker A=y\cdot\cok A=0$.
Then the following hold.
\begin{enumerate}[\rm(a)]
\item
There exists a matrix $B$ over $S$ such that $AB=BA=yE$.
\item
Let $B$ be a matrix as in {\rm(a)}.
Suppose that $y$ is an $S$-regular element.
Let $C$ be a matrix over $S$, and let $x\in S$ be an $S$-regular element.
If $\cok C\in{[\cok A]}_r^{S/(y)}$ for some integer $r>0$, then there is a containment $\cok(xC)\in{[\cok(xA)\oplus\cok(xB)\oplus S/(x)]}_r^{S/(xy)}$.
\end{enumerate}
\end{enumerate}
\end{lem}

\begin{proof}
(1) There is an isomorphism $\cok A\cong\cok B\oplus M$ of $S$-modules.
Let $C$ be a presentation matrix of the $S$-module $M$.
Then we have isomorphisms $\cok A\cong\cok B\oplus\cok C\cong\cok\left(\begin{smallmatrix}B&0\\0&C\end{smallmatrix}\right)$.
Note that $\cok(xB)$ is a direct summand of $\cok\left(x\left(\begin{smallmatrix}B&0\\0&C\end{smallmatrix}\right)\right)$.
Replacing $B$ with $\left(\begin{smallmatrix}B&0\\0&C\end{smallmatrix}\right)$, we may assume that $\cok A\cong\cok B$.
There are exact sequences $F_1\xrightarrow{A}F_0\to N\to0$ and $G_1\xrightarrow{B}G_0\to N\to0$ of $S$-modules with $F_1,F_0,G_1,G_0$ free.
Consider the pullback diagram
$$
\xymatrix@R-.5pc@C+2pc{
&&0&0\\
0\ar[r] & \im A\ar[r]^{\inc} & F_0\ar[r]\ar[u] & N\ar[r]\ar[u] & 0\\
0\ar[r] & \im A\ar@{=}[u]\ar[r]^a & X\ar[u]^f\ar[r]^g & G_0\ar[u]\ar[r] & 0\\
&&\im B\ar[u]^b\ar@{=}[r] & \im B\ar[u]^{\inc}\\
&&0\ar[u]&{\phantom{.}}0.\ar[u]
}
$$
Since $F_0,G_0$ are projective $S$-modules, there are $S$-homomorphisms $s:F_0\to X$ and $t:G_0\to X$ such that the compositions $fs$ and $gt$ are the identity maps.
We have a commutative diagram
$$
\xymatrix@R-.5pc@C+2pc{
\im A\oplus G_0\ar@{=}[r]\ar[d]_\inc^{\left(\begin{smallmatrix}\inc&0\\0&1\end{smallmatrix}\right)} & \im A\oplus G_0\ar[r]^-{(a,t)}_-\cong\ar[d]^{\left(\begin{smallmatrix}\inc&ft\\0&1\end{smallmatrix}\right)} & X\ar[d]^{\left(\begin{smallmatrix}f\\g\end{smallmatrix}\right)} & F_0\oplus\im B\ar[l]_-{(s,b)}^-\cong\ar[d]^{\left(\begin{smallmatrix}1&0\\gs&\inc\end{smallmatrix}\right)} & F_0\oplus\im B\ar@{=}[l]\ar[d]_\inc^{\left(\begin{smallmatrix}1&0\\0&\inc\end{smallmatrix}\right)}\\
F_0\oplus G_0\ar[r]^{\left(\begin{smallmatrix}1&ft\\0&1\end{smallmatrix}\right)}_\cong & F_0\oplus G_0\ar@{=}[r] & F_0\oplus G_0 & F_0\oplus G_0\ar@{=}[l] & F_0\oplus G_0\ar[l]_{\left(\begin{smallmatrix}1&0\\gs&1\end{smallmatrix}\right)}^\cong
}
$$
such that the horizontal maps are isomorphisms.
Remark \ref{9} implies $(F_0\oplus G_0)/x(\im A\oplus G_0)\cong(F_0\oplus G_0)/x(F_0\oplus\im B)$, which shows $\cok(xA)\oplus G_0/xG_0\cong F_0/xF_0\oplus\cok(xB)$.
The assertion now follows.

(2) Let $A$ have $m$ rows.
Then $\cok A=S^m/\im A$ and $\cok(xA)=S^m/\im(xA)$.
The equalities $y\cdot\cok A=0$ and $xy\cdot\cok(xA)=0$ are equivalent to the inclusions $yS^m\subseteq\im A$ and $xyS^m\subseteq\im(xA)$, respectively.
As $x\cdot\im A=\im(xA)$, the first inclusion implies the second.

(3)(a) The assertion is shown similarly to Proposition \ref{7}(2).

(b) Since $y$ is $S$-regular and kills $\cok A$, it is seen that $A,B$ are square matrices of the same size.
We use induction on $r$.
Let $r=1$.
It follows from Proposition \ref{7}(3) that $B$ is a presentation matrix of $\syz_{S/(y)}(\cok A)$ and there is an isomorphism $\syz_{S/(y)}^2(\cok A)\cong\cok A$.
Hence
$$
\cok C\in{[\cok A]}_1^{S/(y)}
={|\{\cok A,\,S/(y),\,\syz_{S/(y)}(\cok A)\}|}^{S/(y)}
={|\{\cok A,\,\cok(y),\,\cok B\}|}^{S/(y)}.
$$
Applying (1) and (2), we observe that $\cok(xC)\in{|\{\cok(xA),\,\cok(xy),\,\cok(xB),S/(x)\}|}^{S/(xy)}$.
Since $\cok(xy)=S/(xy)$, we have $\cok(xC)\in[\cok(xA)\oplus\cok(xB)\oplus S/(x)]_1^{S/(xy)}$.

Now let $r\ge2$.
Then there exists an exact sequence $0\to X\to Y\to Z\to0$ of $S/(y)$-modules with $X\in{[\cok A]}_{r-1}^{S/(y)}$ and $Z\in{[\cok A]}_1^{S/(y)}$ such that $\cok C$ is a direct summand of $Y$ (see \cite[Proposition 2.2(1)]{radius}).
Take presentation matrices $D,F$ of $X,Z$ over $S$, respectively.
The horseshoe lemma yields the commutative diagram in the lower left with exact rows and columns, where $G$ is a matrix of the form $\left(\begin{smallmatrix}D&H\\0&F\end{smallmatrix}\right)$.
This induces the commutative diagram in the lower right with exact rows and columns.
It follows from (2) that $X'=\cok(xD)$, $Y'=\cok(xG)$ and $Z'=\cok(xF)$ are modules over $S/(xy)$.
$$
\xymatrix@R-.5pc@C+.7pc{
0\ar[r] & S^a\ar[r]\ar[r]\ar[d]_D & S^{a+b}\ar[r]\ar[d]_G & S^b\ar[r]\ar[d]_F & 0\\
0\ar[r] & S^c\ar[r]\ar[r]\ar[d] & S^{c+d}\ar[r]\ar[d] & S^d\ar[r]\ar[d] & 0\\
0\ar[r] & X\ar[r]\ar[r]\ar[d] & Y\ar[r]\ar[d] & Z\ar[r]\ar[d] & 0\\
&0&0&0
}\qquad
\xymatrix@R-.5pc@C+.7pc{
0\ar[r] & S^a\ar[r]\ar[r]\ar[d]_{xD} & S^{a+b}\ar[r]\ar[d]_{xG} & S^b\ar[r]\ar[d]_{xF} & 0\\
0\ar[r] & S^c\ar[r]\ar[r]\ar[d] & S^{c+d}\ar[r]\ar[d] & S^d\ar[r]\ar[d] & 0\\
& X'\ar[r]\ar[r]\ar[d] & Y'\ar[r]\ar[d] & Z'\ar[r]\ar[d] & 0\\
&0&0&0
}
$$
Take any element $z\in\ker(xF)$.
The assumption that $x$ is $S$-regular implies $z\in\ker F$.
The left diagram shows that the map $\ker F\to\cok D=X$ induced by the snake lemma is zero, which implies $Hz\in\im D$.
Hence $(xH)z\in\im(xD)$, which shows that the map $\ker(xF)\to\cok(xD)=X'$ induced by the snake lemma is zero.
This gives rise to an exact sequence $0\to X'\to Y'\to Z'\to0$ of $S/(xy)$-modules.
Applying the induction hypothesis, we obtain the containments $X'\in{[\cok(xA)\oplus\cok(xB)\oplus S/(x)]}_{r-1}^{S/(xy)}$ and $Z'\in{[\cok(xA)\oplus\cok(xB)\oplus S/(x)]}_1^{S/(xy)}$, while $\cok(xC)$ is a direct summand of $Y'\oplus(S/(x))^n$ for some $n\ge0$ by (1).
Considering the exact sequence $0\to X'\oplus(S/(x))^n\to Y'\oplus(S/(x))^n\to Z'\to0$, we see that $\cok(xC)\in{[\cok(xA)\oplus\cok(xB)\oplus S/(x)]}_r^{S/(xy)}$.
\end{proof}

\begin{lem}\label{5}
Let $x\in S$ be an $S$-regular element, and let $M\in\mf(x)$.
Then for each integer $n>0$ one has an equality ${[M]}_n^{S/(x)}={|M\oplus S/(x)\oplus\syz_{S/(x)}M|}_n^{S/(x)}$.
\end{lem}

\begin{proof}
Set $\X:={|M\oplus S/(x)\oplus\syz_{S/(x)}M|}^{S/(x)}$.
There is an isomorphism $\syz_{S/(x)}^2M\cong M$ by (2) and (3) of Proposition \ref{7}.
It is observed that $\X={[M]}^{S/(x)}$, and hence
$$
|M\oplus S/(x)\oplus\syz_{S/(x)}M|_n^{S/(x)}
=|\X|_n^{S/(x)}
=|[M]^{S/(x)}|_n^{S/(x)}
\subseteq[[M]^{S/(x)}]_n^{S/(x)}
=[M]_n^{S/(x)}.
$$
Now, pick any $N\in{[M]}_n^{S/(x)}$.
Let us show the containment $N\in{|M\oplus S/(x)\oplus\syz_{S/(x)}M|}_n^{S/(x)}$ by induction on $n$.
The equality $\X={[M]}^{S/(x)}$ given above settles the case $n=1$.
Let $n\ge2$.
Then there exists an exact sequence $0\to A\to B\to C\to0$ of $S/(x)$-modules with $A\in{[M]}_{n-1}^{S/(x)}$ and $C\in{[M]}_1^{S/(x)}$ such that $N$ is a direct summand of $B$.
The induction hypothesis implies $A\in{|M\oplus S/(x)\oplus\syz_{S/(x)}M|}_{n-1}^{S/(x)}$ and $C\in{|M\oplus S/(x)\oplus\syz_{S/(x)}M|}_1^{S/(x)}$.
It follows that $N$ is in ${|M\oplus S/(x)\oplus\syz_{S/(x)}M|}_n^{S/(x)}$, as desired.
\end{proof}

Let $\X$ be a subcategory of $\mod R$.
The {\em dimension} (resp. {\em radius}) of $\X$, denoted by $\dim\X$ (resp. $\radius\X$), is defined to be the infimum of integers $n\ge0$ with $\X={[G]}_{n+1}$ (resp. $\X\subseteq{[G]}_{n+1}$) for some $G\in\mod R$.
Now we can give a proof of our main theorem.

\begin{proof}[\bf Proof of Theorem \ref{0}]
Suppose that $\mf(x_i)={[G_i]}_{d_i+1}^{S/(x_i)}$ for each $1\le i\le n$, where $G_i\in\mf(x_i)$ and $d_i\ge0$.
The assertions (1) and (2) of Proposition \ref{7} imply that for each $1\le i\le n$ there exist square matrices $P_i,Q_i$ such that $P_iQ_i=Q_iP_i=x_iE$, $\ker P_i=\ker Q_i=0$ and $\cok P_i\cong G_i$.
We set
\begin{align*}
H_i&=\cok(x_1\cdots x_{i-1}P_i)\oplus\cok(x_1\cdots x_{i-1}Q_i)\oplus S/(x_1\cdots x_{i-1}),\\
K_i&=H_i\oplus S/(x_1\cdots x_i)\oplus \syz_{S/(x_1\cdots x_i)}H_i.
\end{align*}
Using Lemma \ref{4}(2), we easily check that $H_i$ belongs to $\mf(x_1\cdots x_i)$, and Lemma \ref{5} gives rise to an equality ${[H_i]}_{d_i+1}^{S/(x_1\cdots x_i)}={|K_i|}_{d_i+1}^{S/(x_1\cdots x_i)}$ for all $1\le i\le n$.

Let $M\in\mf(x_1)\ast\cdots\ast\mf(x_n)$.
Put $T_n=M$.
There exist exact sequences
$$
0\to T_i\xrightarrow{f_i} T_{i+1}\to V_i\to0\qquad(1\le i\le n-1)
$$
of $S$-modules with $T_i\in\mf(x_1)\ast\cdots\ast\mf(x_i)$ and $V_i\in\mf(x_{i+1})$.
Setting $M_i=\im(f_{n-1}f_{n-2}\cdots f_{i+1}f_i)\cong T_i$ for each $1\le i\le n-1$, we get a filtration $0=:M_0\subseteq M_1\subseteq\cdots\subseteq M_n:=M$ of $S$-submodules of $M$ such that $M_i/M_{i-1}\cong V_{i-1}\in\mf(x_i)$ for $1\le i\le n$, where $V_0:=M_1\cong T_1\in\mf(x_1)$.
Let $A_i$ be a presentation matrix of $M_i/M_{i-1}$ such that $\ker A_i=0$ for $1\le i\le n$.
By Lemma \ref{3}, we obtain an exact sequence
$$\textstyle
0 \to \bigoplus_{i=1}^n\cok(x_1\cdots x_{i-1}A_i) \to M\oplus(S/(x_1\cdots x_n))^p \to \bigoplus_{i=2}^n(S/(x_i\cdots x_n))^{p_i} \to 0.
$$
As $\cok A_i\in\mf(x_i)={[G_i]}_{d_i+1}^{S/(x_i)}$ for $1\le i\le n$, Lemma \ref{4}(3) implies $\cok(x_1\cdots x_{i-1}A_i)\in{|K_i|}_{d_i+1}^{S/(x_1\cdots x_i)}$ for $1\le i\le n$.
We see that $\bigoplus_{i=1}^n\cok(x_1\cdots x_{i-1}A_i)$ is in ${|\bigoplus_{i=1}^nK_i|}_{d+1}^{S/(x_1\cdots x_n)}$, where $d=\max\{d_1,\dots,d_n\}$.
The above short exact sequence shows $M\in{|\bigoplus_{i=1}^nK_i\oplus\bigoplus_{i=2}^nS/(x_i\cdots x_n)|}_{d+2}^{S/(x_1\cdots x_n)}$.
We conclude that the subcategory $\mf(x_1)\ast\cdots\ast\mf(x_n)$ of $\mod S/(x_1\cdots x_n)$ has radius at most $d+1$.
\end{proof}

\begin{rem}
The above proof of Theorem \ref{0} actually shows the stronger inequality
$$
\size(\mf(x_1)\ast\cdots\ast\mf(x_n))\le\sup\{\dim\mf(x_1),\dots,\dim\mf(x_n)\}+1.
$$
Here, the {\em size} of a subcategory $\X$ of $\mod R$, denoted by $\size\X$, has been introduced in \cite{radius}, which is by definition the infimum of integers $n\ge0$ such that $\X\subseteq{|G|}_{n+1}$ for some $G\in\mod R$.
\end{rem}

For a Cohen--Macaulay local ring $R$ we denote by $\lcm(R)$ the stable category of maximal Cohen--Macaulay $R$-modules, that is, the ideal quotient of the additive category $\cm(R)$ by free modules.
If $R$ is Gorenstein, then $\lcm(R)$ is a triangulated category (see \cite{B}), and the dimension of $\lcm(R)$ in the sense of Rouquier is defined.
For the definition of the dimension of a triangulated category, we refer the reader to \cite{R}.
To show our corollaries, we establish one more lemma.

\begin{lem}\label{6}
Let $R$ be a local hypersurface.
Then there are equalities
$$
\dim\ds(R)=\dim\lcm(R)=\dim\cm(R)=\radius\cm(R).
$$
\end{lem}

\begin{proof}
Since $R$ is a Gorenstein ring of finite Krull dimension, by virtue of \cite[Theorem 4.4.1]{B} there is an equivalence $\ds(R)\cong\lcm(R)$ as triangulated categories.
Hence it holds that $\dim\ds(R)=\dim\lcm(R)$.
As $R$ is a hypersurface, we have $\dim\lcm(R)=\dim\cm(R)=\radius\cm(R)$ by \cite[Proposition 3.5(3)]{dim}.
\end{proof}

Recall that a Cohen--Macaulay local ring $R$ is said to have {\em finite CM-representation type} if there exist only finitely many isomorphism classes of indecomposable maximal Cohen--Macaulay $R$-modules.
When this is the case, it is clear from the definition that $\dim\cm(R)=0$.
Now let us prove our corollaries.

\begin{proof}[\bf Proof of Corollary \ref{1}]
We begin with proving the first assertion of the corollary.
According to Lemma \ref{6}, it suffices to show that
$$
\radius\cm(S/(x_1\cdots x_n))\le\sup\{\dim\cm(S/(x_1)),\dots,\dim\cm(S/(x_n))\}+1.
$$
Fix an integer $1\le i\le n$.
Proposition \ref{7}(5) implies $\cm(S/(x_i))=\mf(x_i)$.
Let $M\in\cm(S/(x_1\cdots x_n))$.
Setting $M_i=(0:_Mx_1\cdots x_i)$, we have a filtration $0=M_0\subseteq\cdots\subseteq M_n=M$ of $S$-submodules of $M$, and $M_i/M_{i-1}$ is an $S/(x_i)$-module.
Note that there is an isomorphism $M_i/M_{i-1}\to(0:_{x_1\cdots x_{i-1}M}x_i)$ given by $\overline{z}\mapsto x_1\cdots x_{i-1}z$ for $z\in M_i$.
The target is a submodule of $M$, and hence it has positive depth.
As the ring $S/(x_i)$ has dimension one, the $S/(x_i)$-module $M_i/M_{i-1}$ is maximal Cohen--Macaulay, that is, $M_i/M_{i-1}\in\cm(S/(x_i))=\mf(x_i)$.
It follows that $M$ belongs to $\mf(x_1)\ast\cdots\ast\mf(x_n)$.
Applying Theorem \ref{0} completes the proof of the first assertion of the corollary.

To show the second assertion of the corollary, suppose that $S/(x_i)$ has finite CM-representation type for all $1\le i\le n$.
Then by Lemma \ref{6} we have $\dim\ds(S/(x_i))=\dim\cm(S/(x_i))=0$ for all $1\le i\le n$.
The first assertion of the corollary implies that $\dim\ds(S/(x_1\cdots x_n))\le1$.
\end{proof}

\begin{proof}[\bf Proof of Corollary \ref{2}]
The inequality $\dim\ds(R)\le1$ is a direct consequence of Corollary \ref{1} and Lemma \ref{6}.
Let $S=\C[\![x,y]\!]$ be a formal power series ring.
For each $f\in S$, the hypersurface $A=S/(f)$ has finite CM-representation type if and only if $f$ belongs to $\P$ after changing variables; see \cite[Theorem (8.10) and Corollary (9.3)]{Y}.
Lemma \ref{6} implies $\dim\ds(A)=\dim\cm(A)$.
Since $A$ is henselian, $\dim\cm(A)=0$ if and only if $A$ has finite CM-representation type by \cite[Proposition 3.7(1)]{dim}.
In conclusion, one has $\dim\ds(R)=0$ if and only if $R\cong S/(f)$ for some $f\in\P$.
The contradiction of this statement is nothing but the assertion of the corollary.
\end{proof}



\end{document}